\documentclass[12pt,reqno]{amsart}
\usepackage{amssymb, amsthm, amsmath}
\usepackage[utf8]{inputenc}
\usepackage{graphicx}
\usepackage{enumitem}
\usepackage{hyperref}
\usepackage{mathtools}
\usepackage{stmaryrd}

\usepackage[margin=.7in,paperwidth=480pt,paperheight=690pt]{geometry}

\newtheorem{theorem}{Theorem}
\newtheorem{lemma}[theorem]{Lemma}
\newtheorem{corollary}[theorem]{Corollary}
\newtheorem{conjecture}[theorem]{Conjecture}
\newtheorem{proposition}[theorem]{Proposition}

\theoremstyle{definition}
\newtheorem{definition}[theorem]{Definition}
\newtheorem{remark}[theorem]{Remark}

\theoremstyle{remark}

\newcommand{\R}{\mathbb{R}}

\newcommand{\Z}{\mathbb{Z}}
\newcommand{\N}{\mathbb{N}}
\newcommand{\PP}{\mathbb{P}}
\newcommand{\EE}{\mathbb{E}}
\newcommand{\cc}{c}
\newcommand{\cb}{\breve{c}}
\newcommand{\tnu}{\nu}
\newcommand{\bnu}{\breve{\nu}}

\renewcommand{\Im}{\operatorname{Im}}
\newcommand{\ind}{\mathbf{1}}
\newcommand{\darrow}{\xrightarrow{\mathrm{d}}}
\newcommand{\varrow}{\xrightarrow{\mathrm{v}}}
\newcommand{\corrarrow}{\xrightarrow{\mathrm{corr}}}
\newcommand{\marrow}{\xrightarrow{\mathrm{m}}}
\newcommand{\LL}{\mathcal{L}}

\DeclareMathOperator{\supp}{supp}

\newcommand{\zero}{\widetilde\gamma}

    \title[On convergence of points]{On convergence of points to limiting processes, with an application to zeta zeros}
    \author{J. Arias de Reyna, B. Rodgers}
    \date{}
    
\begin{document}

\begin{abstract}
This paper considers sequences of points on the real line which have been randomly translated, and provides conditions under which various notions of convergence to a limiting point process are equivalent. In particular we consider convergence in correlation, convergence in distribution, and convergence of spacings between points. We also prove a simple Tauberian theorem regarding rescaled correlations. The results are applied to zeros of the Riemann zeta-function to show that several ways to state the GUE Hypothesis are equivalent. The proof relies on a moment bound of A. Fujii.
\end{abstract}

\maketitle
\thispagestyle{empty}

\section{Introduction}
\label{sec:intro}

\subsection{Statistics of zeros}
The purpose of this note is to fill a gap in the literature, demonstrating that various notions which have been used to describe the conjectured spacing of zeros of the Riemann zeta-function are equivalent. We will phrase our results in a general language, so that they apply to generic sequences of points meeting minimal conditions -- these more general results are presented below, after a discussion of spacings of zeta zeros. Many ideas we use are understood by experts in the theory of point processes, but it can be hard to find a single reference explaining matters, especially in the context of for instance zeta zeros. And as we will see the fact that these notions are equivalent does depend on the zeta zeros satisfying some regularity properties; the equivalences will not be completely generic to any collection of points.

We discuss zeros of the Riemann zeta-function first. As usual we denote by $\rho_n = \sigma_n + i \gamma_n$ for $n\in \mathbb{Z}$ the nontrivial zeros of $\zeta(s)$, where $\sigma_n \in (0,1)$ and $\gamma_n \in \R$ with $\gamma_n \leq \gamma_{n+1}$ and $\gamma_n > 0$ for $n > 0$ and $\gamma_n \leq 0$ otherwise. (The zeros are conjectured to be simple, but we list them according to multiplicity if that is not true.) The Riemann Hypothesis is the claim that $\sigma_n = 1/2$ for all $n$. The ordinates $\gamma_n$ near a value $T \geq 0$ have density roughly $\log T/2\pi$ in the sense that \cite[Thm. 9.4]{T}
\begin{equation}
\label{eq:Riemann_vonMangoldt}
| \{n:\; \gamma_n \in (0,T]\} | \sim T \frac{\log T}{2\pi}.
\end{equation}

The GUE Hypothesis considers the rescaled distribution of ordinates $\gamma_n$ near a random value $t \in [T,2T]$. We write it in the following form:

\begin{conjecture}[GUE Hypothesis]
\label{conj: GUE_correlations}
For any fixed integer $k \geq 1$ and any fixed $f \in C_c(\R^k)$,
\begin{multline}
\label{eq: GUE}
\lim_{T\rightarrow\infty} \frac{1}{T} \int_T^{2T} \sum_{n \in \Z_\ast^k} f\Big(\frac{\log T}{2\pi}(\gamma_{n_1}-t),\, \cdots ,\frac{\log T}{2\pi}(\gamma_{n_k}-t)\Big)\, dt \\
= \int_{\R^k} f(x_1,...,x_k) \det_{1 \leq i, j \leq k}\Big( S(x_i - x_j)\Big)\, d^k x,
\end{multline}
where $S(x) = \frac{\sin \pi x}{\pi x}$.
\end{conjecture}

Here $C_c(\R^k)$ is the space of compactly supported continuous functions and
$$
\Z_\ast^k = \{ (n_1,...,n_k) \in \Z^k: \; n_1, ... ,n_k \textrm{ are all distinct.}\}.
$$

This conjecture is about the random collection of points $\{\frac{\log T}{2\pi}(\gamma_n-t)\}$, where randomness has come from choosing $t$ uniformly from the interval $[T,2T]$. A configuration of points which has been chosen randomly can be formalized by the notion of a point process; see Section \ref{sec: point processes} below. There we explain that the right hand side of \eqref{eq: GUE} for all $k$ are the correlation functions of a limiting point process called the sine-kernel determinantal point process, or \emph{sine-kernel process} for short. When the relation \eqref{eq: GUE} holds for point processes such as $\{\frac{\log T}{2\pi}(\gamma_n-t)\}$ with $t \in [T,2T]$ as $T\rightarrow\infty$, we will say that these point processes \emph{converge in correlation} to the sine-kernel process as $T\rightarrow\infty$. A form of this conjecture first appeared in \cite[Eq. (15)]{M}. Further discussion can be found in \cite{RS}, \cite{KS}, or \cite{R}.

It is natural to consider a rescaling of the ordinates $\gamma_k$ to have mean unit spacing. A popular rescaling is as follows:
$$
\zero_n:= \frac{1}{\pi} \vartheta(\gamma_n), \quad \vartheta(t):= \Im(\log \Gamma(\tfrac{1}{4} + i \tfrac{t}{2})) - \tfrac{t}{2}\log \pi.
$$
By Stirling's formula $\frac{1}{\pi} \vartheta(t) \sim t \frac{\log t}{2\pi}$ (see \cite[Ch. 14]{MV}) and \eqref{eq:Riemann_vonMangoldt} implies $\gamma_n \sim 2\pi n /\log n$, so indeed $\zero_n \sim n$ and the points $\zero_n$ asymptotically have density $1$. It is reasonable to suppose that the points $\{\zero_n-t\}$ for $t \in (0,T]$ also tend in correlation to the sine-kernel process as $T\rightarrow\infty$, that is:
\begin{multline}
\label{eq: GUE_rescaled}
\lim_{T\rightarrow\infty} \frac{1}{T} \int_0^{T} \sum_{n \in \Z_\ast^k} f(\zero_{n_1}-t,\, \cdots , \zero_{n_k}-t)\, dt \\
= \int_{\R^k} f(x_1,...,x_k) \det_{1 \leq i, j \leq k}\Big( S(x_i - x_j)\Big)\, d^k x,
\end{multline}
for all fixed $k\geq 1$ and all fixed $f \in C_c(\R^k)$. One purpose of this note is to record the proof that \eqref{eq: GUE_rescaled} is equivalent to the relation \eqref{eq: GUE} in the GUE Hypothesis and formulate this result in a slightly more general context.

Convergence in correlation is not however the only notion for convergence of point processes. We also have a notion of \emph{convergence in distribution} for the point process $\{\tilde{\gamma}_n-t\}$ for $t \in [0,T]$ to the sine-kernel process. By this we mean for any $k\geq 1$ and any collection of closed intervals $I_1, ..., I_k$ and non-negative integers $\lambda_1, ..., \lambda_k$, the limits 
\begin{equation}
\label{eq:zeta_cylinders}
\lim_{T\rightarrow\infty} \frac{1}{T} \int_0^T \Big\llbracket | \{n:\; \zero_n-t \in I_1\}| = \lambda_1,\, \cdots, | \{n:\; \zero_n-t \in I_k\}| = \lambda_k \Big\rrbracket \, dt
\end{equation}
exist and tend to the probability that $\lambda_1$ points of the sine-kernel process lie in $I_1$, ..., $\lambda_k$ points lie in $I_k$. (This probability is given by \eqref{eq:sine_kernel_cylinders} later. See Sec. \ref{sec: zeta zeros} for further discussion.) Here, following Knuth \cite{K}, we are using $\llbracket P \rrbracket$ as the indicator of a proposition $P$; we have $\llbracket P \rrbracket = 1$ or $0$ depending on whether $P$ is true or false. If $S$ is a finite set, $|S|$ indicates the number of elements in $S$.

Another statistic which has sometimes been studied is the joint distribution of spacings between zeros \cite{O, O2}. In this case one fixes $K \geq 1$ and for a random integer $n$ sampled uniformly from $1$ to $N$ computes the joint distribution of the random variables $(\zero_{n+1}-\zero_n)$,  $(\zero_{n+2}-\zero_n),$ ..., $(\zero_{n+K}-\zero_n)$. We say that \emph{the joint distribution of the spacings between zeros $\{\zero_n\}$ tends to the joint distribution of spacings between points in the sine-kernel process} if for all $K$ this random vector tends as $N\rightarrow\infty$ to the analogous random vector of spacings between points of the sine-kernel process. The proper way to define this analogous quantity for the sine-kernel process is done through the Palm process, discussed in Section \ref{sec: gaps} and Section \ref{subsec: spacings}; for now the reader needs only understand that some limiting joint distribution is defined, via the sine-kernel process.

We show that these notions are equivalent:

\begin{theorem}
\label{thm: main_zetazeros}
The following are equivalent:
\begin{enumerate}[label = (\roman*)]
\item The GUE Hypothesis (that is, for $t \in [T,2T]$ chosen uniformly at random, the point processes $\{\frac{\log T}{2\pi}(\gamma_n-t)\}$ tend in correlation to the sine-kernel process as $T\rightarrow\infty$).
\item The point processes $\{\zero_n-t\}$, for $t \in (0,T]$ chosen uniformly at random, tend in correlation to the sine-kernel process as $T\rightarrow\infty$.
\item The point processes $\{\zero_n-t\}$, for $t \in (0,T]$ chosen uniformly at random, tend in distribution to the sine-kernel process as $T\rightarrow\infty$. (That is \eqref{eq:zeta_cylinders} is always equal to \eqref{eq:sine_kernel_cylinders}.)
\item The joint distribution of spacings between zeros $\{\zero_n\}$ tends to the joint distribution of spacings between points in the sine-kernel process. (See Section \ref{subsec: spacings} for a more precise statement.)
\end{enumerate}
\end{theorem}

In fact the equivalence between \emph{(i)} and \emph{(ii)} depends only on the density of zeros being slowly varying and the rescaling to $\zero_n$ being relatively regular; see Section \ref{sec: respacing}. Moreover the equivalence between \emph{(iii)} and \emph{(iv)} is almost completely generic and relies only on a simple property of the limiting distribution; see Section \ref{sec: gaps}. By contrast, number theorists may find it surprising that technical information about the distribution of zeros is necessary to pass from \emph{(ii)} to \emph{(iii)} (though probabilists may be less surprised). What is required is a moment bound due to A. Fujii \cite{F, F1} for counts of zeros in small random intervals; see Section \ref{subsec:moment_zetazeros}.

\begin{remark}
The averages over the intervals $[T,2T]$ in \emph{(i)} or $(0,T]$ in \emph{(ii)} and \emph{(iii)} could be equivalently replaced in each instance by an average over an interval $[\alpha T, \beta T]$ for fixed $0 \leq \alpha < \beta$, though we do not pursue this more general statement here.
\end{remark}

\subsection{General point processes}

Let us now discuss the general results Theorem \ref{thm: main_zetazeros} is based on. We formulate these results in the language of point processes, reviewed in Section \ref{sec: point processes}. We hope these results will be useful even for readers without an interest in the Riemann zeta-function. Because the results depend on a formalism and notation introduced in Section \ref{sec: point processes} to follow, we content ourselves here to informal statements of the main results with references to more exact statements in the text.

In Section \ref{sec: point processes} we draw out equivalences between various notions of convergence for point processes. In particular, in Theorem \ref{thm:vectorconv} we show under minimal conditions that convergence in distribution of point processes $\xi_\ell$ to a limit point process $\xi$ is equivalent to the convergence of finite dimensional distributions of labelings of $\xi_\ell$ to labelings of $\xi$. Furthermore in Theorem \ref{thm:tail_bound_correlations} we show that convergence in correlation of the processes $\xi_\ell$ to a limiting process $\xi$ implies convergence in distribution, given very mild conditions on the processes $\xi_\ell$ and a tail bound for the probability a large number of points in $\xi$ cluster in a compact interval. In the opposite direction, in Theorem \ref{thm:tail_bound_distributions} we show that uniform moment bounds on counts of points of $\xi_\ell$ are sufficient to see that convergence in distribution implies convergence in correlation. Both Theorems \ref{thm:tail_bound_correlations} and \ref{thm:tail_bound_distributions} are based on the moment method.

In Section \ref{sec: respacing} we prove Tauberian theorems which allow one to show that the correlation functions of points have the same limit under sufficiently regular rescaling. In particular Theorem \ref{thm: respace_equiv} is a generalization of the claim that the correlation functions points $\{\tfrac{\log T}{2\pi}(\gamma_n-t)\}$ and the points $\{\zero_n-t\}$ will have the same limit as long as a limit exists.

Finally, in Section \ref{sec: gaps} we show in Theorem \ref{thm:spacing_to_palm} that in the case that the limiting point processes is simple, convergence in distribution of a point process induced by randomly translating points is equivalent to the convergence of finite dimensional distributions of gaps. In this case the distribution of gaps is described through Palm processes; we recall the definition and provide the necessary background in this section.

We finally describe how these results apply to zeta zeros in Section \ref{sec: zeta zeros}.

\subsection{Acknowledgments} B.R. received partial support from an NSERC grant. We thank Jeff Lagarias, Anurag Sahay, and an anonymous referee for comments and feedback.

\section{Point processes: fundamental background}
\label{sec: point processes}

The proper framework to treat the problems in this paper is that of point processes. Intuitively, a point process on the real line is just a configuration of points $\{...,x_{-1},x_0,x_1,...\} \subset \mathbb{R}$ which has been laid down randomly. A point processes in turn is best thought of as a special sort of random measure. We therefore recall the formalism of random measures first.

In this paper we will work with random measures of the real line $\mathbb{R}$, where $\R$ is endowed with its usual topology. There is a more general theory of random measures defined on any Polish space (that is, any topological space which is separable, metrizable, and complete); see \cite{Ka1} for a comprehensive account.

\subsection{Locally finite measures}
\label{subsec: measures}

In this subsection we set up the space of measures from which a random measure will be drawn. To begin with, we recall that $\mathbb{R}$ is a Polish space, and let $\mathcal{R}$ be the Borel $\sigma$-algebra of $\mathbb{R}$. A Borel measure $\mu$ on $\mathbb{R}$ is said to be locally finite (or Radon) if $\mu(B) < +\infty$ for any bounded $B \in \mathcal{R}$. We let $\mathfrak{M}$ be the class of all locally finite Borel measures on $\mathbb{R}$. 

We endow $\mathfrak{M}$ with the \emph{vague topology} $\tau$, generated by the maps $\nu \mapsto \int f\, d\nu$ for $f \in C_c(\mathbb{R})$, so that for a sequence of $\nu_1, \nu_2,...\in \mathfrak{M}$ we have $\nu_\ell$ converges to $\nu \in \mathfrak{M}$ (written $\nu_\ell \varrow \nu$) if and only if we have the convergence of real numbers $\int f \, d\nu_\ell \rightarrow \int f\, d\nu$ for all $f \in C_c(\mathbb{R})$. (Here $C_c(\mathbb{R})$ denotes the collection of continuous and compactly supported functions.) It is known that $\mathfrak{M}$ with the vague topology is a Polish space (see \cite[15.7.7]{Ka1}). We label $\mathcal{M}$ the Borel $\sigma$-algebra of $\mathfrak{M}$ generated by the vague topology $\tau$. 

Later we will need the following alternative characterization of convergence in the vague topology. As usual we let $\partial B$ denote the boundary of a set $B \subset \mathbb{R}$.

\begin{proposition} (Portmanteau theorem for vague convergence)
\label{prop:vague}
For measures $\nu, \nu_1, \nu_2,... \in \mathfrak{M}$, the following are equivalent: 
\begin{enumerate}[label = (\roman*)]
	\item $\nu_\ell \varrow \nu$.
	\item $\nu_\ell(B) \rightarrow \nu(B)$ for all bounded $B \in \mathcal{R}$ such that $\nu(\partial B) = 0$.
	\item $\limsup_{\ell\rightarrow \infty} \nu_\ell(F) \leq \nu(F)$ for all bounded closed $F \in \mathcal{R}$, and \sloppy $\liminf_{\ell\rightarrow\infty} \nu_\ell(G) \geq \nu(G)$ for all bounded open $G \in \mathcal{R}$.
\end{enumerate}
\end{proposition}

\begin{proof}
This is \cite[15.7.2]{Ka1}.
\end{proof}

\subsection{Random measures}
\label{subsec: random measures}

A \emph{random measure on $\mathbb{R}$} is a random element of $(\mathfrak{M},\mathcal{M})$. We now state some important facts about random measures and especially their convergence.

As usual we refer to random elements of $(\mathbb{R},\mathcal{R})$ as random variables, and random elements of $(\mathbb{R}^d,\mathcal{R}^d)$ as random vectors. 

In general, we will use the following notation for distributional convergence: if $X, X_1, X_2,...$ are random elements of a metric space $S$ with Borel $\sigma$-algebra $\mathcal{S}$ we say that $X_\ell$ tends to $X$ in distribution and write $X_\ell \darrow X$ if there is the convergence of real numbers 
$$
\EE\, F(X_\ell) \rightarrow \EE\, F(X),
$$
for any $F \in C_b(S)$, where $C_b(S)$ is the collection of bounded continuous functions $F: S\rightarrow\mathbb{R}$. (If $T$ is a continuous parameter, we define $X_T \darrow X$ in the same way.) This definition thus characterizes convergence in distribution of random variables, random vectors, and random measures.

The following result gives useful alternative criteria for convergence in distribution of random measures.

\begin{theorem}[Generalized Portmanteau theorem for random measures]
\label{thm:distconv}
For $\mu, \mu_1, \mu_2,...$ random measures on $\mathbb{R}$, the following statements are equivalent:
\begin{enumerate}[label = (\roman*)]
\item We have convergence in distribution of random measures, $\mu_\ell \darrow \mu$.
\item For all $f \in C_c(\mathbb{R})$, we have the convergence in distribution of real-valued random variables, 
$$
\int f\, d\mu_\ell \darrow \int f\, d\mu.
$$
\end{enumerate}
In addition if almost surely $\mu(\{a\})=0$ for any point $a \in \mathbb{R}$, (i) and (ii) are equivalent to
\begin{enumerate}[label = (\roman*)]
\setcounter{enumi}{2}
\item For any $k$ and any collection of closed intervals $I_1 = [a_1,b_1],...,I_k=[a_k,b_k]$, we have the convergence in distribution of random vectors
$$
(\mu_\ell(I_1),...,\mu_\ell(I_k)) \darrow (\mu(I_1),...,\mu(I_k)) .
$$
\end{enumerate}
\end{theorem}

\begin{proof}
This is a special case of Theorem 4.2 in \cite{Ka1}. 
\end{proof}

\subsection{Point configurations}
\label{subsec: point configs}

We let $\mathfrak{N}$ denote the class of locally finite Borel measures on $\mathbb{R}$ which are non-negative integer valued (i.e. if $\xi \in \mathfrak{N}$ then $\xi(B) \in \N_{\geq 0}$ for any bounded $B \in \mathcal{R}$). Clearly $\mathfrak{N} \subset \mathfrak{M}$. We let $\mathfrak{N}$ inherit the topology of $\mathfrak{M}$ and we label by $\mathcal{N}$ the Borel $\sigma$-algebra of $\mathfrak{N}$ in this topology.

For any $x\in \mathbb{R}$, let $\delta_x \in \mathfrak{N}$ be the point mass at $x$ defined by
$$
\delta_x(B) = \begin{cases} 1 & \textrm{if}\; x \in B \\ 0 & \textrm{if}\; x \notin B. \end{cases}
$$
It is clear that locally finite configurations of points $\{x_n\}_{n\in \Z}$ can be identified with $\xi \in \mathfrak{N}$ via $\xi = \sum \delta_{x_n}$.

Conversely we identify $\xi \in \mathfrak{N}$ with locally finite configurations of points through the following proposition. For ease of notation we define the following constraint: we say that $\xi$ is \emph{infinite in both directions} if $\xi(-\infty,0) = \xi(0,\infty) = \infty$. 

\begin{proposition}
\label{prop:meas_to_points}
For $\xi \in \mathfrak{N}$ with $\xi$ infinite in both directions, define for $n > 0$
$$
x_n(\xi):= \inf\{x > 0:\, \xi(0,x] \geq n\},
$$
and for $n \leq 0$,
$$
x_n(\xi):= \sup\{x \leq 0:\, \xi(x,0] \geq -n+1\}.
$$
Then
\begin{equation}
\label{eq:pointmass}
\xi = \sum_{n \in \Z} \delta_{x_n(\xi)},
\end{equation}
and for all $n \in \mathbb{Z}$, we have $x_n(\xi)$ is continuous in the topology $\tau$ of $\mathfrak{N}$ at all elements $\xi \in \mathfrak{N}$ except those for which $\xi(\{0\})  \geq 1$.
\end{proposition}

Thus under this labeling,
$$
\cdots \leq x_{-1}(\xi) \leq x_0(\xi) \leq 0 < x_1(\xi) \leq x_2(\xi) \leq \cdots
$$ 
are the point masses of $\xi$ (listed with multiplicity). The condition that $\xi$ is infinite in both directions ensures that $x_n(\xi)$ is a well-defined real number of all $n$.

\begin{proof}
We first verify \eqref{eq:pointmass}. We must show for all $f \in C_c(\mathbb{R})$,
\begin{equation}
\label{eq:pointmass2}
\int_\R f\, d\xi = \sum_n f(x_n(\xi)).
\end{equation}
Suppose $f$ is supported in the interval $[-A,A]$ and $\xi(-A,0] = N$. Then for any $n\in \mathbb{Z}$, if $x_n(\xi) > -A$ we have,
$$
x_n(\xi) = \inf\{x > -A: \xi(-A,x] \geq N+n\}.
$$
Furthermore the left-hand side of \eqref{eq:pointmass2} is then a Riemann-Stieljes integral $\int f\, d\alpha$, where $\alpha$ is the step function
$$
\alpha(x):= \xi(-A,x] = \sum_{\substack{n \in \Z \\ x_n(\xi) > -A}} \mathbf{1}_{[x_n(\xi),\infty)}(x),
$$
and \eqref{eq:pointmass2} follows immediately.
		
It remains to verify the continuity of $x_n(\xi)$ at $\xi$ without point masses at the origin. We begin by treating the case that $n > 0$. Suppose $\xi_\ell \varrow \xi$ and $\xi(\{0\}) = 0$. Then for arbitrary $x' < x_n(\xi)$ we have $\xi[0,x'] < n$ so by the Portmanteau theorem, $\limsup_{\ell\rightarrow\infty} \xi_\ell[0,x'] < n$ and thus $x_n(\xi_\ell) > x'$ for sufficiently large $\ell$. On the other hand for arbitrary $x'' > x_n(\xi)$ we have $\xi(0,x'') \geq n$ so $\liminf_{\ell\rightarrow\infty} \xi_\ell(0,x'') \geq n$ and thus $x_n(\xi_\ell) < x''$ for sufficiently large $\ell$. Hence for sufficiently large $\ell$,
$$
x' < x_n(\xi_\ell) < x''.
$$
As $x'$ and $x''$ can be chosen arbitrarily close to $x_n(\xi)$ we have $x_n(\xi_\ell) \rightarrow x_n(\xi)$.

The case for $n \leq 0$ is treated in the same way.	
\end{proof}

\begin{remark}
Note from \eqref{eq:pointmass}, for sets $E \in \mathcal{R}$, we have $\xi(E)$ is the number of point $x_n(\xi)$ which lie in $E$.
\end{remark} 

\begin{remark} 
The functions $x_n(\xi)$ need not be continuous at those $\xi$ for which $\xi(\{0\}) \geq 1$. For imagine a point process $\xi$ evolves in time, beginning with no point mass at the origin. Suppose $\xi$ is then continuously varied by shifting its point mass corresponding to $x_1(\xi)$ to the left, while leaving all other point masses fixed. When the point mass being shifted crosses the origin of $\R$, this causes a discontinuous relabeling of points. Indeed, it may be seen that there is no way to label the points of $\xi$ which is always continuous. (Suppose that there were such a way, and label the points of the lattice $\Z$. Then continuously translate each point in this lattice to the left until $\Z$ is recovered again.)
\end{remark}

\subsection{Point processes}
\label{subsec: point processes}

A \emph{point process on} $\mathbb{R}$ is a random element of $(\mathfrak{N},\mathcal{N})$. Given point processes $\xi, \xi_1, \xi_2,...$ we say that $\xi_\ell$ tends to $\xi$ in distribution as $\ell\rightarrow\infty$ if $\xi_\ell\darrow \xi$ as random measures. We will therefore use the same notation for convergence in distribution as before. 

In addition to the characterization of distributional convergence implied by Theorem \ref{thm:distconv} we also have
\begin{theorem}
\label{thm:vectorconv}
For $\xi, \xi_1, \xi_2,...$ point processes on $\mathbb{R}$, suppose that $\xi$ and $\xi_1,\xi_2,..$ are all infinite in both directions, and that almost surely $\xi(\{0\}) = 0$. Then the following are equivalent:
\begin{enumerate}[label = (\roman*)]
\item $\xi_\ell \darrow \xi$.
\item For any positive integer $K$ we have the convergence in distribution of random vectors
$$
(x_{-K}(\xi_\ell),...,x_K(\xi_\ell)) \darrow (x_{-K}(\xi),...,x_K(\xi)).
$$
\end{enumerate}
\end{theorem}

\begin{proof}
The continuous mapping theorem for random elements (see \cite[Thm. 4.27]{Ka2}) tells us that if $S$ and $T$ are metric spaces and $G: S \rightarrow T$ is a continuous function on a set $C \subset S$, then the distributional convergence of random elements $Y_\ell$ to $Y$ implies the distributional convergence of random elements $G(Y_\ell)$ to $G(Y)$ as long as $Y$ belongs to the continuity set $C$ almost surely.

The implication that \emph{(i)} implies \emph{(ii)} is an immediate application of this. For note that if $C \subset \mathfrak{N}$ is the set of point processes $\nu$ for which $\nu(\{0\}) = 0$, then each $x_n$ (and therefore the vector in \emph{(ii)}) will be continuous on $C$, and we have supposed $\xi$ lies in this continuity set $C$ almost surely.

To see that \emph{(ii)} implies \emph{(i)}, note first that \emph{(ii)} is equivalent to 
$$
(x_n(\xi_\ell))_{n \in \Z} \darrow (x_n(\xi))_{n \in \Z},
$$
where $(x_n(\xi_\ell))_{n \in \Z}$ and $(x_n(\xi))_{n \in \Z}$ are elements of the product space $\R^\Z$ endowed with the (metrizable) product topology. (For discussion of the product topology and a proof that it is metrizable see \cite[Ex. 16, Ch. 12.2]{Ro}.) 

We rely on condition \emph{(ii)} of the generalized Portmanteau theorem for random measures and the representation \eqref{eq:pointmass} of $\xi$ as a sum of point masses, so that to see $\xi_\ell \darrow \xi$ we need only show that
\begin{equation}
\label{eq:needtoshow_linearstats}
\sum_{n \in \Z} f(x_n(\xi_\ell)) \darrow \sum_{n \in \Z} f(x_n(\xi)) 
\end{equation}
for all $f \in C_c(\R)$. For a given $f \in C_c(\R)$, the reader may verify that the map $G: \R^\Z \rightarrow \R$ defined by
$$
G(y) := \sum_{n\in \Z} f(y_n)
$$
\sloppy is continuous at $y = (y_n)_{n \in \Z}$ as long as $\lim_{|n|\rightarrow\infty}|y_n| = \infty,$ owing to the compact support of $f$. But because $\xi$ is locally finite, we have $\lim_{|n|\rightarrow\infty}|x_n(\xi)|= \infty$ always, and so an application of the continuous mapping theorem implies \eqref{eq:needtoshow_linearstats} and we are done.
\end{proof}

\subsection{Correlations and moments}
\label{subsec:corr}
We now introduce correlation measures associated with point processes. For $\xi$ a point process on $\R$, if there exists a locally finite measure $\rho_k$ on $\mathbb{R}^k$ such that
\begin{equation}
\label{eq:corrmeas_def}
\EE \sum_{n \in \Z_\ast^k} f\big( x_{n_1}(\xi),...,x_{n_{k}}(\xi)\big) = \int_{\R^k} f(x_1,...,x_k)\, d\rho_k(x_1,...,x_k),
\end{equation}
for all $f \in C_c(\R^k)$, then $\rho_k$ is referred to as the \emph{$k$-th correlation measure} (or \emph{$k$-th joint intensity}) of $\xi$. Such a measure will exist as long as $\xi(I) \in \LL^k$ for all compact intervals $I$ (see \cite[Prop. 3.2]{L2} or \cite[Thm. A.1]{LaRo}).

If $\xi, \xi_1, \xi_2,...$ are point processes on $\R$ each with $k$-th correlation measures defined for all $k$, we say that $\xi_\ell$ \emph{tends in correlation} to $\xi$ and write $\xi_\ell \corrarrow \xi$ if for all $k$ and all $f \in C_c(\R^k)$ we have the convergence of real numbers
\begin{equation}
\label{eq:corr_general}
\EE \sum_{n \in \Z_\ast^k} f(x_{n_1}(\xi_\ell),...,x_{n_k}(\xi_\ell)) \rightarrow \EE \sum_{n \in \Z_\ast^k} f(x_{n_1}(\xi),...,x_{n_k}(\xi)),
\end{equation}
as $\ell\rightarrow\infty$. (That is, if $\xi_\ell$ has $k$-th correlation measure $\rho_k^\ell$ and $\xi$ has $k$-th correlation measure $\rho_k$, we say $\xi_\ell \corrarrow \xi$ if on $\R^k$ there is the vague convergence of measures $\rho_k^\ell \varrow \rho_k$ for all $k$.)

We also use the following terminology: for $X, X_1, X_2, \ldots $ random variables, we say $X_\ell$ tends to $X$ in moments if $X, X_\ell \in \LL^k$ for all $k\geq 1$, $\ell$ and $\EE\, X_\ell^k \rightarrow \EE\, X^k$ for all $k$. We write $X_\ell \marrow X$ in this case. (As usual, we use the notation $\LL^k$ to denote the space of random variables with finite $k$-th absolute moment.)

\begin{proposition}
\label{prop:correlations_moments}
For $\xi, \xi_1, \xi_2,...$ point processes on $\R$, suppose that $\xi(I)$, $\xi_\ell(I) \in \LL^k$ for all $k\geq 1, \ell$ and all compact intervals $I$. Then the following statements are equivalent:
\begin{enumerate}[label = (\roman*)]
\item We have the convergence in correlation of point processes $\xi_\ell \corrarrow \xi$.
\item For all $f \in C_c(\R)$, we have the convergence in moments of random variables,
$$
\int f \, d\xi_\ell \marrow \int f \, d\xi.
$$
\end{enumerate}
\end{proposition}

\begin{proof}
We proceed through a series of equivalences.

\emph{Step 1:} First we note that $\xi_\ell \corrarrow \xi$ if and only if
\begin{equation}
\label{eq:corr_tensor}
\EE \sum_{n \in \Z_\ast^k} f_1(x_{n_1}(\xi_\ell))\cdots f_k(x_{n_k}(\xi_\ell)) \rightarrow \EE \sum_{n \in \Z_\ast^k} f_1(x_{n_1}(\xi))\cdots f_k(x_{n_k}(\xi)),
\end{equation}
for all $k\geq 1$ and all $f_1,...,f_k \in C_c(\R)$. Clearly $\xi_\ell \corrarrow \xi$ implies \eqref{eq:corr_tensor}, by considering correlations with a test function $f(x_1,...,x_k) = f_1(x_1)\cdots f_k(x_k)$. On the other hand such test functions are dense in $C_c(\R^k)$, so by a standard approximation argument \eqref{eq:corr_tensor} entails \eqref{eq:corr_general} for arbitrary $f \in C_c(\R^k)$.

\emph{Step 2:} Next we note that \eqref{eq:corr_tensor} holding for all $k\geq 1$ and $f_1,...,f_k \in C_c(\R)$ is equivalent to convergence of mixed moments
\begin{equation}
\label{eq:mixed_moments}
\EE \prod_{i=1}^k \sum_{n\in \Z} f_i(x_n(\xi_\ell)) \rightarrow \EE \prod_{i=1}^k \sum_{n\in \Z} f_i(x_n(\xi)),
\end{equation}
for all $k\geq 1$ and $f_1,...,f_k \in C_c(\R)$. 

We can see this in the following way. Given a locally finite configuration of points $\{x_n\}$, let
$$
M_k(f_1,...,f_k) = \prod_{i=1}^k \sum_{n\in \Z} f_i(x_n), \quad C_k(f_1,...,f_k) = \sum_{n \in \Z_\ast^k} f_1(x_{n_1})\cdots f_k(x_{n_k}),
$$
be mixed products of linear statistics and $k$-point correlation sums respectively. Then we will show
$$
M_k(f_1,...,f_k) - C_k(f_1,...,f_k)
$$
is a linear combination of sums of the sort $C_j(g_1,...,g_j)$ for $j\leq k-1$ and $g_i \in C_c(\R)$ for all $i$. Because $M_1(f) = C_1(f)$, by taking expectations and using induction this will show that the claim that \eqref{eq:corr_tensor} holds for all $k\geq 1$ is equivalent to the claim that \eqref{eq:mixed_moments} holds for all $k\geq 1$.

It thus remains in this step to show $M_k(f_1,...,f_k)-C_k(f_1,...,f_k)$ is a linear combination of $j$-point correlation sums for $j \leq k-1$. We prove this by induction in $k$, with $k=1$ following from $M_1(f) = C_1(f)$. Note that
\begin{multline}
\label{eq:corr_induction}
\Big(\sum_{n \in \Z} f(x_n)\Big) C_j(g_1,...,g_j) \\
= C_{j+1}(g_1,...,g_j,f) + C_j(g_1 f, g_2,...,g_j)+ \cdots + C_j(g_1,...,g_j f).
\end{multline}
Therefore if
$$
M_{k-1}(f_1,...,f_{k-1}) = C_{k-1}(f_1,...,f_{k-1}) + \mathcal{E}
$$
where $\mathcal{E}$ is a linear combination of $j$-point correlation sums with $j \leq k-2$, then
\begin{multline*}
M_k(f_1,...,f_k) = \Big(\sum_{n\in \Z} f_k(x_n)\Big) M_{k-1}(f_1,...,f_{k-1}) \\
= C_k(f_1,...,f_k) + C_{k-1}(f_1 f_k,..., f_{k-1}) + \\
\cdots + C_{k-1}(f_1,...,f_{k-1} f) + \Big(\sum_{n\in \Z} f_k(x_n)\Big) \mathcal{E},
\end{multline*}
and by \eqref{eq:corr_induction}, the last term on the right will be a linear combination of $j$-point correlation sums with $j\leq k-1$. Thus inductively our claim is verified.

\emph{Step 3:} Finally we show for any $k\geq 1$ we have convergence of mixed moments \eqref{eq:mixed_moments} for all $f_1,...,f_k \in C_c(\R)$ if and only if we have convergence of moments
\begin{equation}
\label{eq:unmixed_moment}
\EE \Big( \int f d\xi_\ell \Big)^k \rightarrow \EE \Big( \int f d\xi \Big)^k,
\end{equation}
for all $f \in C_c(\R)$. That \eqref{eq:mixed_moments} implies \eqref{eq:unmixed_moment} is clear by setting $f = f_1 = \cdots = f_k$. On the other hand, one may go from \eqref{eq:unmixed_moment} to \eqref{eq:mixed_moments} by using a multilinear polarization identity:
$$
S_1 \cdots S_k = \frac{1}{k!} \sum_{\varepsilon \in \{0,1\}^k} (-1)^{k-\varepsilon_1 - \cdots -\varepsilon_k} (\varepsilon_1 S_1+ \cdots + \varepsilon_k S_k)^k,
$$
setting $S_i = \int f_i\, d\xi_\ell$. (More information about multilinear polarization can be found in e.g. the survey \cite{Th}.)

Combining these steps we thus see the equivalence of (i) and (ii).
\end{proof}

From the method of moments we thus have

\begin{theorem}
\label{thm:tail_bound_correlations}
Suppose $\xi$ is a point process on $\R$ and for any compact interval $I \subset \R$ we have 
\begin{equation}
\label{eq:limiting_tail_bound}
\PP(\xi (I) \geq t) \leq C e^{-ct}, \quad \textrm{for all}\; t \geq 0,
\end{equation}
for positive constants $C,c$ which depend on $I$. 

Suppose $\xi_1, \xi_2,...$ are point processes on $\R$ satisfying $\xi_\ell(I) \in \LL^k$ for all compact intervals $I$ and all $\ell, k \geq 1$. If $\xi_\ell \corrarrow \xi$ then we also have $\xi_\ell \darrow \xi$.
\end{theorem}

\begin{proof}
From Proposition \ref{prop:correlations_moments}, the claim $\xi_\ell \corrarrow \xi$ implies $\int f\, d\xi_\ell \marrow \int f\, d\xi$ for all $f \in C_c(\R)$. But \eqref{eq:limiting_tail_bound} implies 
$$
\PP\Big(\Big|\int f\, d\xi\Big| \geq t\Big) \leq C e^{-ct},
$$ for positive constants $C$ and $c$ (which depend on $f$). But then $\EE |\int f\, d\xi|^k \leq k! A^k$ for some constant $A$ (depending on $C$ and $c$), and so the random variable $\int f\, d\xi$ is determined by its moments (see \cite[Sec. 30]{B}). Thus $\int f\, d\xi_\ell \marrow \int f\, d\xi$ implies $\int f\, d\xi_\ell \darrow \int f\, d\xi$ by \cite[Thm. 30.2]{B}.
\end{proof}

\begin{remark}
Clearly from the proof the condition \eqref{eq:limiting_tail_bound} can be relaxed to the random variables $\xi(I)$ always being determined by their moments.
\end{remark}

In the opposite direction,

\begin{theorem}
\label{thm:tail_bound_distributions}
Suppose $\xi_1, \xi_2,...$ are point processes on $\R$ and for any fixed compact interval $I \subset \R$ and any fixed $k \geq 0$ we have
\begin{equation}
\label{eq:sup_moment_bound}
\sup_\ell \EE\, |\xi_\ell(I)|^k \leq C_{I,k}
\end{equation}
for a constant $C$ which depends on $I$ and $k$.

If $\xi$ is a point process on $\R$ and $\xi_\ell \darrow \xi$, then we have $\xi_\ell \corrarrow \xi$.
\end{theorem}

\begin{proof}
By Theorem \ref{thm:distconv} $\xi_\ell \darrow \xi$ implies for any $f\in C_c(\R)$ that $\int f \, d\xi_\ell \darrow \int f\, d\xi$. \eqref{eq:sup_moment_bound} implies that for any $k$ one has 
$$
\sup_\ell\; \EE \,\Bigl|\int f\, d\xi_\ell\Bigr|^{k+\epsilon} < \infty
$$
for some $\epsilon > 0$, so by Corollary of Theorem 25.12 in \cite{B}, $\int f\, d\xi \in \LL^k$ for all $k$ and $\int f \, d\xi_\ell \marrow \int f\, d\xi$. Hence the $k$-th correlation measures of $\xi$ are defined for all $k$ and $\xi_\ell \corrarrow \xi$ by Proposition \ref{prop:correlations_moments}.
\end{proof}

\begin{remark}
Conditions similar to \eqref{eq:limiting_tail_bound} and \eqref{eq:sup_moment_bound} are necessary for Theorems \ref{thm:tail_bound_correlations} and \ref{thm:tail_bound_distributions} to be true. As remarked upon the condition \eqref{eq:limiting_tail_bound} can be weakened slightly, but clearly \eqref{eq:sup_moment_bound} must be true for the limits of correlations of $\xi_\ell$ to even exist. 

One may construct point processes which converge in correlation but do not converge in distribution or vice-versa. A simple construction is as follows: take integer-valued random variables $X_\ell$ which converge in moments but not in distribution or vice-versa (see \cite[Sec. 30]{B} for further discussion). Now let $\xi_\ell$ be the point process with a quantity $X_\ell$ points at the origin, and no other points. Then $\xi_\ell$ will converge in correlation but not distribution, or vice-versa. (One can modify this construction if point processes infinite in both direction are desired; e.g. place $X_\ell$ points at the origin and one point deterministically at each integer other than the origin.)
\end{remark}

\section{Respacing correlations}
\label{sec: respacing}

In this section we show that different methods of rescaling a deterministic sequence of points do not affect limiting correlation functions, at least as long as the rescaling is sufficiently regular. The proof depends on Tauberian ideas.

Recall that a function $L(x)$ is said to be \emph{slowly varying} \cite[Sec. IV.2]{Ko} if it is measurable, eventually positive, and for any fixed $r > 0$ satisfies
\begin{equation}
\label{eq:slowvar}
\frac{L(rx)}{L(x)} \rightarrow 1, \quad \textrm{as } x\rightarrow\infty.
\end{equation}
It is known \cite[Thm. 2.2]{Ko} that if $L$ is slowly varying, then for any fixed $0 < b < B < \infty$, the limit in \eqref{eq:slowvar} holds uniformly for $r \in [b,B]$.

We consider a sequence of positive real numbers $\{c_n\}_{n \geq 1}$ listed in non-decreasing order (repitition is allowed) with slowly varying density: if
$$
N_c(T) := |\{n:\; c_n \in (0,T]\}|
$$
then we will limit our attention to the case
\begin{equation}
\label{eq:N_count}
N_c(T) \sim T \cdot L(T), \quad \textrm{as } T\rightarrow\infty
\end{equation}
for a slowly varying function $L(T)$.

Our first lemma verifies that different ranges of averaging for correlation functions are equivalent.

\begin{lemma}
\label{lem: avg_equiv}
Let $\{c_n\}_{n\geq 1}$ be a non-decreasing sequence of positive real numbers satisfying \eqref{eq:N_count}. For any $k\geq 1$, let $\mu_k$ be a measure on $\R^k$. Then the following are equivalent:
\begin{enumerate}[label = (\roman*)]
\item For all $\eta \in C_c(\R^k)$,
$$
\lim_{T\rightarrow\infty} \frac{1}{T} \int_0^T \sum_{n \in \Z_\ast^k} \eta( L(t)(c_{n_1}-t), \, \cdots , L(t)(c_{n_k}-t))\, dt = \int_{\R^k} \eta(x) \, d\mu_k(x).
$$
\item For all $\eta \in C_c(\R^k)$,
$$
\lim_{T\rightarrow\infty} \frac{1}{T} \int_T^{2T} \sum_{n \in \Z_\ast^k} \eta( L(T)(c_{n_1}-t), \, \cdots , L(T)(c_{n_k}-t))\, dt = \int_{\R^k} \eta(x) \,d\mu_k(x).
$$
\end{enumerate}
\end{lemma}

\begin{remark}
Note that in \emph{(i)}, in addition to an average from $[0,T]$, the dilation factor $L(t)$ varies throughout the integral. In \emph{(ii)} the dilation factor $L(T)$ remains constant.
\end{remark}

\begin{proof}
The proof depends upon the fact that $L(t)$ is slowly varying and so for $t \in [T,2T]$, changing $L(t)$ to $L(T)$ will have negligible affect on correlation sums.

Let us first show that \emph{(ii)} is equivalent to the claim that for all $\eta \in C_c(\R^k)$,
\begin{multline}
\label{eq:avg_varying}
\int_T^{2T} \sum_{n \in \Z_\ast^k} \eta( L(t)(c_{n_1}-t), \, \cdots , L(t)(c_{n_k}-t))\, dt \\= T \int_{\R^k} \eta(x) \, d\mu_k(x) + o_{T\rightarrow\infty}(T).
\end{multline}
Indeed, suppose \emph{(ii)} is true. For all $t \in [T,2T]$, we have $L(t) = L(T) + o_{T\rightarrow\infty}(L(T))$. Moreover in the sum which appears in \emph{(ii)}, we may restrict to $n$ such that $|c_{n_i}-t| \leq A/L(T)$ for all $i$, supposing $\eta$ is supported in a box $K = [-A,A]^k$ for some constant $A$. For such $n$ then $L(t)(c_{n_i}-t) = L(T)(c_{n_i}-t) + o_{T\rightarrow\infty}(1)$, and so by continuity of $\eta$,
\begin{multline*}
\eta( L(T)(c_{n_1}-t), \, \cdots , L(T)(c_{n_k}-t)) \\ = \eta( L(t)(c_{n_1}-t), \, \cdots , L(t)(c_{n_k}-t)) + o_{T\rightarrow\infty}(1).
\end{multline*}
Thus
\begin{multline*}
\sum_{n \in \Z_\ast^k} \eta( L(T)(c_{n_1}-t), \, \cdots , L(T)(c_{n_k}-t)) \\
= \sum_{n \in \Z_\ast^k} \eta( L(t)(c_{n_1}-t), \, \cdots , L(t)(c_{n_k}-t)) \\
+ o_{T\rightarrow\infty}\Big(\sum_{n \in \Z_\ast^k} \ind_K( L(T)(c_{n_1}-t), \, \cdots , L(T)(c_{n_k}-t))\Big).
\end{multline*}
If we integrate $t \in [T,2T]$, we have that \emph{(ii)} implies left hand side will be equal to $T \int \eta d\mu_k + o(T)$. But likewise, by majorizing the function $\ind_K$ by a continuous and compactly supported non-negative function on $\R^k$, \emph{(ii)} implies the integral of the error term on the right hand side will be $o(T)$. This implies \eqref{eq:avg_varying}. In the opposite direction, one may pass from \eqref{eq:avg_varying} for all $\eta$ to \emph{(ii)} in the same way.

Thus we need only show that \eqref{eq:avg_varying} is equivalent to \emph{(i)}. But this is just an equivalence between full and dyadic averaging; indeed if \emph{(i)} is true, one may replace $T$ with $2T$ and then from the resulting integral on $[0,2T]$ subtract the original integral on $[0,T]$. This implies \eqref{eq:avg_varying}. In the opposite direction, if \eqref{eq:avg_varying} is true, then for a given $\eta$ the left hand side of \eqref{eq:avg_varying} must uniformly be $O(T)$ for all $T$. (It will be this for sufficiently large $T$ by the limiting relation, and then for $T$ smaller than this value, the bound of $O(T)$ follows by compactness.) Hence,
\begin{multline}
\label{eq:full_avg}
\int_0^T \sum_{n \in \Z_\ast^k} \eta( L(t)(c_{n_1}-t), \, \cdots , L(t)(c_{n_k}-t))\, dt \\
= \sum_{v = 0}^\infty \int_{T/2^{v+1}}^{T/2^v} \sum_{n \in \Z_\ast^k} \eta( L(t)(c_{n_1}-t), \, \cdots , L(t)(c_{n_k}-t))\, dt.
\end{multline}
Let $V$ be arbitrarily large but fixed and consider the sums $v < V$ and $v \geq V$ seperately. In the former case we use the limiting form of \eqref{eq:avg_varying} and in the latter case we use an upper bound. The result is that for arbitrary $V$ the above is
\begin{multline*}
(1-2^{-V}) T \int \eta \,d\mu_k + O(2^{-V} T) + o_{T\rightarrow\infty}(T) \\= T \int \eta \,d\mu_k + O(2^{-V} T) + o_{T\rightarrow\infty}(T)
\end{multline*}
As $V$ is arbitrary this implies the sought after estimate that \eqref{eq:full_avg} is equal to $T \int \eta \, d\mu_k + o_{T\rightarrow\infty}(T)$.
\end{proof}

\begin{remark}
We note that in claim \emph{(ii)} of this Lemma, the average over the interval $[T,2T]$ could be replaced by any $[\alpha T, \beta T]$ for $0 < \alpha < \beta$ and the Lemma would remain true. Under certain conditions the average can be set to $[0,T]$; for instance uniform upper bounds $\int_0^T |N_c(t + 1/L(T))-N_c(t)|^k\, dt \leq C_k T$ for absolute constants $C_k$ and all sufficiently large $T$ will have this implication. (See Corollary \ref{cor:Fujii_cor1} for such an estimate in the context of zeta zeros.) In any case, we do not require replacing the integral from $[T,2T]$ with an integral from $[0,T]$ in what follows and leave the proof of this remark to the reader.
\end{remark}

We now rescale the points $c_n$ themselves to have mean unit density. This is accomplished by choosing a rescaling function $\phi$ satisfying
\begin{equation}
\label{eq:phi_matching}
\phi(t) \sim t \, L(t).
\end{equation}
The reader may check that then $\phi(c_n) \sim n$, and $|\{n:\; \phi(c_n) \in [0,T]\}| \sim T$. 

We will suppose $\phi$ satisfies the following mild regularity conditions: $\phi(t)$ is differentiable and if $\Phi(t):= \phi(t)/t$ then
\begin{equation}
\label{eq:small_deriv}
\Phi'(t)/\Phi(t) = o(1/t).
\end{equation}
In fact for any slowly varying function $L$, functions $\phi$ satisfying these conditions will exist. This can be verified using Karamata's characterization of slowly varying functions \cite[Thm. 2.2 (iii)]{Ko}; we leave the details to interested readers. In any case we will see that $\phi(t) = \vartheta(t)/\pi$ or $\phi(t) = t \tfrac{\log(t)}{2\pi}$ are such functions for the ordinates $\{\gamma_n\}$.

\begin{theorem}
\label{thm: respace_equiv}
Let $\{c_n\}_{n\geq 1}$ and $\mu_k$ be as in Lemma \ref{lem: avg_equiv}, and suppose the rescaling function $\phi$ is differentiable and satisfies \eqref{eq:phi_matching} and \eqref{eq:small_deriv}. Then conditions (i) or (ii) of Lemma \ref{lem: avg_equiv} are equivalent to:
\begin{enumerate}[label = (\roman*)]
\setcounter{enumi}{2}
\item For all $\eta \in C_c(\R^k)$
\begin{equation*}
\lim_{T\rightarrow\infty} \frac{1}{T} \int_0^T \sum_{n\in \Z_\ast^k} \eta\big( \phi(c_{n_1})-t, \cdots, \phi(c_{n_k})-t\big)\, dt = \int_{\R^k} \eta(x)\, d\mu_k(x).
\end{equation*}
\end{enumerate}
\end{theorem}

\begin{proof}
We begin by considering points $c_n$ and $t$ such that either $|\phi(c_n)-\phi(t)| \leq A$ or $|L(t)(c_n-t)| \leq A$ for some fixed constant $A$. A crude bound reveals that in either case $t/2 \leq c_n \leq 2t$ for sufficiently large $t$. We claim that we have then
\begin{equation}
\label{eq:phi_to_L}
\phi(c_n)-\phi(t) = L(t)(c_n-t)+o_{t\rightarrow\infty}(1).
\end{equation}
If $c_n = t$ this is clear so we may suppose $c_n \neq t$. Then for some $x$ in between $t$ and $c_n$,
\begin{equation*}
\label{eq:deriv_approx}
\frac{\phi(c_n)-\phi(t)}{c_n-t} = \phi'(x) = \Phi(x) + x\Phi'(x) = (1+o(1))\Phi(x).
\end{equation*}
But $t/2 \leq x \leq 2t$ and $\Phi$ is slowly varying so the above is
\begin{equation*}
= (1+o(1))\Phi(t) = (1+o(1))L(t).
\end{equation*}
If $|c_n-t| \leq A/L(t)$ then \eqref{eq:phi_to_L} follows immediately. On the other hand if $|\phi(c_n)-\phi(t)| \leq A$ then \eqref{eq:phi_to_L} follows by swapping the factor $(1+o(1))$ to the other side.

With \eqref{eq:phi_to_L} established, suppose \emph{(i)} of Lemma \ref{lem: avg_equiv}. Then as in the proof of Lemma \ref{lem: avg_equiv}, if $\eta$ is supported in a box $K = [-1,1]^k$, we have
\begin{multline*}
\sum_{n \in \Z_\ast^k} \eta\Big(L(t)(c_{n_1}-t),\cdots, L(t)(c_{n_k}-t)\Bigr) \\
= \sum_{n\in \Z_\ast^k} \eta\big( \phi(c_{n_1})-\phi(t), \cdots, \phi(c_{n_k})-\phi(t)\big) \\
+ o_{t\rightarrow\infty}\Big(\sum_{n \in \Z_\ast^k} \ind_K\Big(L(t)(c_{n_1}-t),\cdots, L(t)(c_{n_k}-t))\Big).
\end{multline*}
Integrating in $t$ and bounding the error term using \emph{(i)} for a test function majorizing $\ind_K$, we have
\begin{equation*}
\int_0^T \sum_{n\in \Z_\ast^k} \eta\big( \phi(c_{n_1})-\phi(t), \cdots, \phi(c_{n_k})-\phi(t)\big) \, dt = T \int_{\R^k} \eta\, d\mu_k + o(T).
\end{equation*}
By differencing, the same estimate can be obtained for an integral from $T$ to $2T$ rather than an integral from $0$ to $T$.

Now note $\phi'(t) = \Phi(t)+t\Phi'(t) \sim \Phi(t) \sim \Phi(T)$ uniformly for $t \in [T,2T]$ (meaning the ratio of any two sides of an asymptotic is $1+o_{T\rightarrow\infty}(1)$ for all $t \in [T,2T]$). Hence
\begin{multline*}
\int_T^{2T} \sum_{n\in \Z_\ast^k} \eta\big( \phi(c_{n_1})-\phi(t), \cdots, \phi(c_{n_k})-\phi(t)\big) \phi'(t)\, dt \\
= T \Phi(T) \int_{\R^k} \eta\, d\mu_k + o(T\Phi(T)) \\
= (\phi(2T)-\phi(T)) \int_{\R^k} \eta\, d\mu_k + o_{T\rightarrow\infty}(\phi(2T)-\phi(T)),
\end{multline*}
with the last line following from the slow variation of $\Phi$. We sum dyadically and make a change of variable $\tau = \phi(t)$ to obtain
\begin{equation*}
\int_0^{\phi(T)} \sum_{n\in \Z_\ast^k} \eta\big( \phi(c_{n_1})-\tau, \cdots, \phi(c_{n_k})-\tau\big) dt = \phi(T) \int_{\R^k} \eta\, d\mu_k + o(\phi(T)).
\end{equation*}
Since any sufficiently large number lies in the image of $\phi(T)$, this proves \emph{(iii)}.

The demonstration that $\emph{(iii)}$ implies $\emph{(i)}$ of Lemma \ref{lem: avg_equiv} is essentially identical, but proceeds in the opposite direction.
\end{proof}

\section{The distribution of consecutive spacings}
\label{sec: gaps}

In this section we consider a non-decreasing sequence of positive real numbers $\{\cc_n\}_{n\geq 1}$ satisfying
\begin{equation}
\label{eq:c_tilde_count}
N_{\tilde{c}}(T) = |\{n:\, \cc_n \in (0,T]\}| \sim T,
\end{equation}
and the corresponding point processes $\xi_T$ with configurations given by $\{\cc_n-t\}_{n \geq 1}$ with $t$ a random variable chosen uniformly from $[0,T]$. We will show that the convergence of the processes $\xi_T$ to a limiting process, satisfying minimal conditions, is governed by the convergence of the consecutive spacings $\cc_{n+1}-\cc_n, \, \cc_{n+2}-\cc_n, \,...$ to a limiting distribution, where $n$ is sampled uniformly from $1$ to $N$ and $N\rightarrow\infty$. 

\subsection{Palm processes}
\label{subsec:palmprocesses}

Before we can discuss the consecutive spacing distribution of a limiting process, we must define the notion of a Palm process. References include \cite{Ma} (which has notation closest to our own), \cite[Ch. 10]{Ka1}, \cite[Ch. 11]{Ka2} and \cite[Sec. 4.2.6-4.2.7]{AnGuZe}. In what follows we give the essential background and cite results we will use.

For a real number $s$ and a set $B \subset \R$, let us use the notation $B+s = \{x+s:\, x \in B\}$ for right translation. On measures this induces a \emph{left} translation operator $\theta^s: \mathfrak{M} \rightarrow \mathfrak{M}$, defined by $\theta^s \nu(B) = \nu(B+s)$ for Borel measures $\nu$. A random measure $\mu$ on $\R$ is said to be \emph{stationary} if its distribution is invariant under translations: $\PP(\theta^s \mu \in E) = \PP(\mu \in E)$ for all $E \in \mathcal{M}$ and all $s \in \R$.

If $\xi$ is a stationary point process on $\R$, then there is a constant $I_\xi \in [0,\infty]$ such that $\EE\, \xi(B) = I_\xi \cdot |B|$ for all Borel measurable sets $B$, where $|B|$ is the Lebesgue measure of $B$. $I_\xi$ is referred to as the \emph{intensity} of the stationary point process $\xi$. Note that a stationary point process $\xi$ on $\R$ having intensity $c$ is equivalent to the claim $d\rho_1(x) = c\, dx$. It is clear that if $\xi$ is stationary and $I_\xi > 0$ then almost surely $\xi$ is infinite in both directions. (Further discussion can be found just after the proof of Lemma 11.1 in \cite{Ka2})

\begin{definition}
\label{def:palm}
If $\xi$ is a stationary point process on $\R$ with intensity $I_\xi \in (0,\infty)$, we define the Palm process $\xi^0$ associated to $\xi$ by
\begin{equation}
\label{eq:palm_def}
\PP(\xi^0 \in E) = \frac{1}{I_\xi} \EE\, \int_0^1 \mathbf{1}_E(\theta^s\xi)\, d\xi(s),
\end{equation}
for any $E \in \mathcal{N}$.
\end{definition}

By stationarity we may equivalently define
$$
\PP(\xi^0 \in E) = \frac{1}{I_\xi |B|} \EE\, \int_B \mathbf{1}_E(\theta^s\xi)\, d\xi(s),
$$
for any Borel $B \in \mathcal{R}$ with $|B|>0$. This is the definition given in \cite[Eq. (23)]{Ma}. (Indeed, the reader may directly verify that \eqref{eq:palm_def} does indeed define a probability measure on $(\mathfrak{N}, \mathcal{N})$.)

Note that \eqref{eq:palm_def} implies 
\begin{equation}
\label{eq:palm_expec}
\EE f(\xi^0) = \frac{1}{I_\xi\cdot |B|} \EE\, \int_B f(\theta^s\xi)\, d\xi(s),
\end{equation}
for any bounded measurable $f: \mathfrak{N} \rightarrow \R$ and any Borel $B \in \mathcal{R}$ with $|B| \geq 0$.

For $\xi$ satisfying the conditions of Definition \ref{def:palm}, we have almost surely $\xi^0(\{0\}) \geq 1$. For we have
\begin{multline*}
\PP(\xi^0(\{0\}) \geq 1) = \frac{1}{I_\xi} \EE \int_0^1 \big\llbracket \theta^s \xi(\{0\}) \geq 1 \big\rrbracket\, d\xi(s) \\
= \frac{1}{I_\xi} \EE \int_0^1 \big\llbracket \xi(\{s\}) \geq 1 \big\rrbracket \, d\xi(s) = \frac{1}{I_\xi} \EE\, \xi[0,1] = 1,
\end{multline*}
with the second to last step following because $\xi(\{s\}) \geq 1$ at each atom of $\xi$. (Use \eqref{eq:pointmass} to decompose $\xi$ into atoms, and recall Knuth bracket notation.)

Indeed the Palm process $\xi^0$ can be thought of as $\xi$ conditioned on having an atom at the origin. More on this analogy can be found in the references given and it will be elaborated by the examples studied below.

We have the following convergence criterion

\begin{proposition}
\label{prop:palm_conv}
Let $\xi, \xi_1, \xi_2,...$ be stationary point processes on $\R$ with intensities $I_\xi, I_{\xi_1}, I_{\xi_2},...$ all finite and nonzero. Then any two of the following statements imply the third:
\begin{enumerate}[label = (\roman*)]
\item $I_{\xi_\ell}\rightarrow I_\xi$
\item $\xi_\ell \darrow \xi$
\item $\xi_\ell^0 \darrow \xi^0$.
\end{enumerate}
\end{proposition}

For us, the implication that \emph{(i)} and \emph{(iii)} imply \emph{(ii)} will be the most important consequence of this result.

\begin{proof}
This is \cite[Prop. 7]{Ma}. (Note that, as \cite{Ma} states, this is a special case of a more general result first proved in \cite[Thm. 6.1]{Ka3}.)
\end{proof}

A point process on $\R$ is said to be \emph{simple} if 
$$
\PP(\exists x \in \R: \, \xi(\{x\}) \geq 2) = 0.
$$
(The above event, it may be checked, is in $\mathcal{N}$ due to the separability of $\R$.) If $\xi$ is simple then it follows immediately from the definition $\xi^0$ is simple also.

In analogy with the definition for point processes on $\R$, a sequence of random variables $(X_j)_{j\in \Z}$ is said to be \emph{stationary} if for any positive $K$ and any shift $a \in \Z$, the random vector $(X_{-K},...,X_K)$ has the same distribution as $(X_{-K+a},...,X_{K+a})$.

\begin{proposition}
\label{prop:stationary_diff}
If $\xi$ is a stationary simple point process on $\R$ with intensity $I_\xi$ nonzero and finite, the sequence of random variables
$$
(x_{j+1}(\xi^0)-x_{j}(\xi^0))_{j \in \Z}
$$
is stationary.
\end{proposition}

\begin{proof}
This follows exactly as in the proof of Lemma 4.2.42 of \cite{AnGuZe}. (Note that in the statement of that Lemma the point process is assumed to be determinantal, but by following the proof given there, one sees that the conclusions of the above Proposition follow with no determinantal hypothesis.
\end{proof}

\subsection{Consecutive spacings in the limit}

We now apply the machinery of Palm processes to the consecutive spacings of a sequence of positive real numbers $\{\cc_n\}_{n\geq 1}$ satisfying condition \eqref{eq:c_tilde_count}, and recall for $T > 0$, the point process $\xi_T$ is defined as the random measure
\begin{equation}
\label{eq:nu_def}
\tnu_t = \sum_{n\geq 1} \delta_{\cc_n-t},
\end{equation}
where $t \in (0,T]$ is random and chosen uniformly.

$\xi_T$ is not stationary (nor is it infinite in both directions). We rectify this in the following way: take $N = N_{\cc}(T)$, and for $n \in \Z$ define
$$
\cb_n = \cb_n(T) = \cc_m + rT, \quad \textrm{where}\; n = m + rN,\; \textrm{for}\; 1 \leq m \leq N,\; r\in \Z.
$$
Thus the sequence $\{\cb_n\}_{n\in \Z}$ consists of the points $\{\cc_1, \cc_2, ..., \cc_N\} \subset (0,T]$ tessellated by translations of size $T$ to cover all of $\R$.

Define the point process $\breve{\xi}_T$ to be the random measure
$$
\bnu_t = \sum_{n \in \Z} \delta_{\cb_n-t},
$$
where $t \in (0,T]$ is random and chosen uniformly. By the periodicity of $\{\cb_n\}_{n\in \Z}$, the point process $\breve{\xi}_T$ is stationary.

The following Proposition confirms $\breve{\xi}_T$ is a good approximation to $\xi_T$ as $T\rightarrow\infty$.

\begin{proposition}
\label{prop:xi_to_period}
Let $\{\cc_n\}_{n\geq 1}$ be any locally finite collection of positive points and define $\xi_T$ and $\breve{\xi}_T$ as above. For a point process $\xi$, we have $\xi_T \darrow \xi$ if and only if $\breve{\xi}_T \darrow \xi$.
\end{proposition}

\begin{proof}
For $\xi_T\darrow \xi$, then by Theorem \ref{thm:distconv} we have $\int f\, d\xi_T \darrow \int f\, d\xi$ for any $f \in C_c(\R)$. But if $\supp f \subset [-\alpha, \beta]$, then
$$
\int f\, d\nu_t = \int f\, d\bnu_t, \quad \textrm{for}\; t \in (\alpha, T -\beta).
$$
Hence as $T\rightarrow\infty$, the above relation holds with probability $1 - o(1)$, and $\int f\, d\xi_T - \int f\, d\breve{\xi}_T$ converges to $0$ in probability and therefore in distribution. Thus $\int f\, d\breve{\xi}_T \darrow \int f\, d\xi$ as claimed. The opposite implication is the same.
\end{proof}

The Palm process $\breve{\xi}^0_T$ associated to $\breve{\xi}_T$ has a nice characterization:

\begin{proposition}
\label{prop:period_palm}
For $N\geq 1$ define the point process $\breve{\Xi}_T$ as the random measure
\begin{equation}
\label{eq:up_def}
\breve{\upsilon}_n = \sum_{j \in \Z} \delta_{\cb_j-\cb_n},
\end{equation}
where $n$ is a random integer chosen uniformly from $1$ to $N = N_{\cc}(T)$. Then for $\{\cc_n\}_{n\geq 1}$ any locally finite collection of positive points, the point processes $\breve{\xi}^0_T$ and $\breve{\Xi}_T$ are identically distributed.
\end{proposition}

\begin{proof}
For bounded measurable $f: \mathfrak{N} \rightarrow \R$, we will use \eqref{eq:palm_expec} with $B = (0,T]$. Before proceeding note that $\breve{\xi}_T(0,T] = N$ almost surely, so $I_{\breve{\xi}_T}\cdot T = \EE\, \breve{\xi}_T(0,T] = N$. Hence
\begin{align}
\label{eq:expec_fxibreve}
\notag \EE f(\breve{\xi}^0_T) &= \frac{1}{N} \EE \int_0^T f(\theta^s \breve{\xi}_T)\, d\breve{\xi}_T(s) \\
\notag &= \frac{1}{N} \frac{1}{T} \int_0^T \int_0^T f(\theta^s \bnu_t)\, d\bnu_t(s)\, dt \\
&= \frac{1}{N} \frac{1}{T} \int_0^T \sum_{n:\, \cb_n -t \in (0,T]} f(\theta^{\cb_n-t} \bnu_t)\, dt.
\end{align}
But
$$
\theta^{\cb_n-t} \bnu_t = \sum_{j\in \Z} \delta_{(\cb_j-t)-(\cb_n-t)} = \sum_{j \in \Z} \delta_{\cb_j - \cb_n}
$$
and because a sum over $\cb_n-t \in (0,T]$ involves a sum over a complete set of $N$ consecutive $\cb_n$, by re-indexing the sum in $j$ we have
$$
\sum_{n:\, \cb_n -t \in (0,T]} f\Big( \sum_{j \in \Z} \delta_{\cb_j - \cb_n}\Big) = \sum_{n=1}^N f\Big( \sum_{j \in \Z} \delta_{\cb_j - \cb_n}\Big),
$$
for all $t$. Substituting this in \eqref{eq:expec_fxibreve} gives 
$$
\EE f(\breve{\xi}^0_T) = \frac{1}{N} \sum_{n=1}^N f\Big( \sum_{j \in \Z} \delta_{\cb_j - \cb_n}\Big)
$$
and the claim follows.
\end{proof}

We may now treat convergence to a limit.

\begin{proposition}
\label{prop:discrete_palm_conv}
Let $\{\cc_n\}_{n\geq 1}$ be a non-decreasing sequence of points satisfying \eqref{eq:c_tilde_count} and define $\xi_T$ by \eqref{eq:nu_def}. For $\xi$ a stationary point process on $\R$ with intensity $I_\xi = 1$, the following are equivalent:
\begin{enumerate}[label = (\roman*)]
	\item $\xi_T \darrow \xi$
	\item $\breve{\Xi}_T \darrow \xi^0$.
\end{enumerate}
\end{proposition}

\begin{proof}
From Proposition \ref{prop:xi_to_period}, we have \emph{(i)} is equivalent to $\breve{\xi}_T\darrow \xi$. But we have of the intensity, $I_{\breve{\xi}_T} = N_{\cc}(T)/T \rightarrow 1$, so Proposition \ref{prop:palm_conv} implies $\breve{\xi}_T\darrow \xi$ is equivalent $\breve{\xi}_T^0 \darrow \xi^0$. But by Proposition \ref{prop:period_palm} this is the same thing as \emph{(ii)}.
\end{proof}

\begin{remark}
If one defines $\Xi_T$ as the random measure
$$
\upsilon_n = \sum_{j\geq 1} \delta_{\cc_j - \cc_n}
$$
where $n$ is a random integer chosen uniformly from $1$ to $N = N_{\cc}(T)$, then an argument similar to what we have used in Proposition \ref{prop:xi_to_period} shows that \emph{(ii)} is equivalent to $\Xi_T \darrow \xi^0$. But we will not require this in what follows.
\end{remark}

If $\xi$ is simple Proposition \ref{prop:palm_conv} can be stated somewhat more memorably.

\begin{theorem}
\label{thm:spacing_to_palm}
Let $\{\cc_n\}_{n\geq 1}$, $\xi_T$, and $\xi$ be as in Proposition \ref{prop:discrete_palm_conv} with the additional condition that $\xi$ is simple. Then the following are equivalent:
\begin{enumerate}[label = (\roman*)]
	\item $\xi_T \darrow \xi$
	\item For any positive $K$ we have the convergence of random vectors
$$
(\cc_{n+1}-\cc_n,\, ... \,, \cc_{n+K}-\cc_n) \darrow (x_1(\xi^0),..., x_K(\xi^0)),
$$
where $n$ is random and uniformly chosen from $1$ to $N$ with $N\rightarrow\infty$.
\end{enumerate}
\end{theorem}

Of course \emph{(ii)} is intuitively close to the condition $\breve{\Xi}_T \darrow \xi^0$ treated above. Our proof will show that in the case that $\xi$ is simple this intuition is correct.

\begin{proof}
Let use first suppose \emph{(i)} and demonstrate \emph{(ii)}. From Proposition \ref{prop:discrete_palm_conv} we have that \emph{(i)} implies $\breve{\Xi}_T \darrow \xi^0$.

For an integer $n$, introduce the notation $n' = n'(T)$ to denote the largest integer such that $\cb_{n'} = \cb_n$. (Thus if $\cb_{n+1} > \cb_n$ then $n' = n$.)

We claim that because $\xi$ is simple, \emph{(i)} implies that
\begin{equation}
	\label{eq:nprime_equal_n}
	\PP(n' = n) \rightarrow 1
\end{equation}
for $n$ a random integer chosen uniformly from $1$ to $N = N_{\cc}(T)$, and where $T\rightarrow\infty$. For using Proposition \ref{prop:period_palm},
$$
\PP(n' \neq n) = \PP(\breve{\Xi}_T(\{0\})\geq2) \rightarrow \PP(\xi^0(\{0\})\geq 2) = 0.
$$

But note likewise from Proposition \ref{prop:period_palm}, $(x_j(\breve{\Xi}_T))_{j \in \Z}$ has the same distribution as $(\cb_{n'+j}-\cb_{n'})_{j \in \Z}$ for $n$ taken uniformly from $1$ to $N$. Thus from Theorem \ref{thm:vectorconv} we have
$$
(\cb_{n'-K}-\cb_{n'},\,....\,,\cb_{n'+K}-\cb_{n'}) \darrow (x_{-K}(\xi^0),...,x_K(\xi^0)),
$$
for any positive $K$, which by \eqref{eq:nprime_equal_n} is the same as
$$
(\cb_{n-K}-\cb_{n},\,....\,,\cb_{n+K}-\cb_{n}) \darrow (x_{-K}(\xi^0),...,x_K(\xi^0)).
$$
As $K$ is fixed it is clear that this is the same as
\begin{equation}
	\label{eq:c_gaps}
	(\cc_{n-K}-\cc_n,...,\cc_{n+K}-\cc_n) \darrow (x_{-K}(\xi^0),...,x_K(\xi^0)),
\end{equation}
which obviously contains \emph{(ii)}, since $N_{\cc}(T)$ ranges over all integers as $T$ goes to $\infty$.

In the other direction, note first by taking linear combinations of entries, \emph{(ii)} implies 
\begin{multline*}
(\cc_{n+1}-\cc_n, \cc_{n+2}-\cc_{n+1},..., \cc_{n+K}-\cc_{n+K-1}) \\
\darrow (x_1(\xi^0)-x_0(\xi^0) , x_2(\xi^0)-x_1(\xi^0) , ... , x_K(\xi^0)-x_{K-1}(\xi^0)),
\end{multline*}
for all positive $K$.

In turn this implies 
\begin{multline*}
(\cc_{n-K}-\cc_{n-K-1},...,\cc_{n+K}-\cc_{n+K-1}) \\
\darrow(x_{-K}(\xi^0)-x_{-K-1}(\xi^0) , ... , x_K(\xi^0)-x_{K-1}(\xi^0))
\end{multline*}
by shifting the average in $n$ and using Proposition \ref{prop:stationary_diff}. (Adopt the convention $c_{n-K} = 0$ for $n-K \leq 0$.) Again taking linear combinations of entries, this implies \eqref{eq:c_gaps}.

The rest of the path from \emph{(ii)} to \emph{(i)} goes exactly as above in the opposite order, with the additional observation that $\PP(n'=n) \rightarrow 1$ (for $n$ chosen uniformly from $1$ to $N\rightarrow\infty$) when \emph{(ii)} is assumed because
\[
\PP(\cc_{n+1}-\cc_n = 0) \rightarrow \PP(x_1(\xi^0) = 0) = 0.\qedhere
\]
\end{proof}

\begin{remark}
It is  possible to work out a similar but slightly more complicated criterion even when $\xi$ is not simple by incorporating equiprobable shifts of size $0, 1, 2, ..., \xi^0(\{0\})-1$ to the coordinates considered on the right hand side of \emph{(ii)}, but we do not pursue this here.
\end{remark}

\section{Applying to zeta zeros}
\label{sec: zeta zeros}

Let us now discuss an application of these results to zeta zeros and prove Theorem \ref{thm: main_zetazeros}.

\subsection{The sine-kernel process}
\label{subsec:sinekernel}

We have not yet defined the sine-kernel process mentioned at the start of the paper.

The sine-kernel process $\varsigma$ is a simple point process on $\R$ with correlation measures given by
\begin{equation}
	\label{eq:sine_corr}
	d\rho_k(x_1,...,x_k) = \det_{1 \leq i, j \leq k}\Big( S(x_i - x_j)\Big)\, d^k x.
\end{equation}
That a simple point process exists with these correlation measures is proved in \cite[Sec. 4.3.5]{HK} or \cite[p.~230]{AnGuZe}. Much in the same way that it is possible for two random variables to share the same moments, it is possible that two distinct point processes have the same correlation measures, but it is known that the correlation measures above uniquely characterize the sine-kernel process. (That is for any point process $\xi$ with correlation measures \eqref{eq:sine_corr}, the probabilities $\PP(\xi \in E)$ will be determined by this information, for any $E \in \mathcal{N}$.) For a proof that these correlation measures characterize the process, see \cite[Lemma 4.2.6]{HK}; the proof is closely related to Theorem \ref{thm:tail_bound_correlations}. Indeed we also have the following tail bound.
\begin{proposition}
\label{prop:sinekernel_tailbound}
For any interval $I \subset \R$,
\begin{equation}
	\label{eq:sinekerneltail}
	\PP(\varsigma(I)\geq t) \leq C e^{-ct}
\end{equation}
for positive constants $C$, $c$ which depend on $I$.
\end{proposition}  
\begin{proof}
This is entailed by \cite[Lemma 4.2.6]{HK}.
\end{proof}

\begin{remark}
The proof of \cite[Lemma 4.2.6]{HK} reveals the authors show that for any positive $c$, there exists a positive constant $C = C_c$ such that \eqref{eq:sinekerneltail} holds. In fact even more than this is true; see \cite[Lemma 16]{TaVu} for a subgaussian bound.
\end{remark}

\begin{remark}
The proof of \cite[Lemma 4.2.6]{HK} also contains a demonstration via Hadamard's inequality of the fact that
\begin{equation}
\label{eq:sineineq}
0 \leq \det_{1 \leq i, j \leq k}\Big( S(x_i - x_j)\Big) \leq S(0)^k = 1,
\end{equation}
which we will use in a moment.
\end{remark} 

For any $k$, any collection of intervals $I_1,...,I_k$ and non-negative integers $\lambda_1,...,\lambda_k$, one may compute the value of
\begin{equation}
\label{eq:sine_kernel_cylinders}
\PP(\varsigma(I_1) = \lambda_1,..., \varsigma(I_k) = \lambda_k),
\end{equation}
though formulas for this probability are more complicated than the expression \eqref{eq:sine_corr} for correlations. \cite[Thm. 2]{S} is a general formula based on Fredholm determinants from which these probabilities can be extracted. Alternatively the following formulation may be seen as more concrete.

\begin{proposition}
\label{prop:occupation_prop}
For mutually disjoint intervals $I_1,...,I_k$, and non-negative integers $\lambda_1,...,\lambda_k$
$$
\PP(\varsigma(I_1) = \lambda_1,..., \varsigma(I_k) = \lambda_k) = \frac{1}{\lambda!} \sum_{n \in \Z^k_{\geq 0}} \frac{(-1)^{|n|}}{n!} C(I, \lambda + n),
$$
where the series converges absolutely and here we have used the shorthand $\lambda = (\lambda_1,...,\lambda_k)$ with $\lambda!  = \lambda_1! \cdots \lambda_k$ (and likewise for the vector $n$), and we use the notation $|n| = n_1 + \cdots + n_k$ and for $a \in \Z_{\geq 0}^k$
$$
C(I,a) = \int_{I_1^{a_1}\times \cdots \times I_k^{a_k}} \det_{1 \leq i,j \leq A} \Big( S(x_i - x_j)\Big)\, d^A x,
$$
with $A = |a|$.
\end{proposition}

\begin{remark}
Knowing $\PP(\varsigma(I_1) = \lambda_1,..., \varsigma(I_k) = \lambda_k)$ for mutually disjoint intervals $I_1,...,I_k$ is of course enough to find such probabilities for $I_1,...,I_k$ which may have overlaps.
\end{remark}

\begin{proof}
Our proof follows that of \cite[Thm. 2]{S} until the final steps, but we include details for the convenience of readers. Note that
\begin{multline}
\label{eq:gen_series}
\prod_{j=1}^k z_j^{\varsigma(I_j)} = \prod_{j=1}^k (1 + z_j-1)^{\varsigma(I_j)} \\
= \prod_{j=1}^k \sum_{a_j=0}^\infty \frac{\varsigma(I_j)(\varsigma(I_j)-1)\cdots (\varsigma(I_j) - (a_j-1))}{a_j !} (z_j-1)^{a_j}.
\end{multline}
The sums in the last line are actually finite, over only $a_j \leq \varsigma(I_j)$ but it is more convenient to write them this way.

We note that for any $a \in \Z_{\geq 0}^k$,
$$
\EE \prod_{j=1}^k \varsigma(I_j)(\varsigma(I_j)-1)\cdots (\varsigma(I_j) - (a_j-1)) = C(I,a).
$$
This follows by applying the definition \eqref{eq:corrmeas_def} of correlation measures to a function $f$ given by an indicator function of the set $I_1^{a_1}\times \cdots \times I_k^{a_k}$. Furthermore, from \eqref{eq:sineineq}, we have
\begin{equation}
\label{eq:C_bound}
0 \leq C(I,a) \leq |I_1|^{a_1} \cdots |I_k|^{a_k}.
\end{equation}

Thus taking the expectation of \eqref{eq:gen_series}, and using the shorthand $z^\lambda = z_1^{\lambda_1}\cdots z_k^{\lambda_k}$ and $E(I,\lambda) =  \PP(\varsigma(I_1) = \lambda_1,..., \varsigma(I_k) = \lambda_k)$ we have for $z_1,...,z_k$ in a neighborhood of the origin,
\begin{multline*}
\sum_{\lambda \in \Z_{\geq 0}^k} E(I,\lambda) z^\lambda = \sum_{a \in \Z_{\geq 0}^k} \frac{C(I,a)}{a!} (z_1-1)^{a_1} \cdots (z_k-1)^{a_k} \\
= \sum_{\lambda, n \in \Z_{\geq 0}^k} \frac{C(I,\lambda+n)}{\lambda!\, n!} z^\lambda (-1)^{|n|},
\end{multline*}
where the interchange of expectation and summation, as well as the rearrangement of power series, is justified by dominated convergence and Proposition \ref{prop:sinekernel_tailbound} and \eqref{eq:C_bound}. Pairing coefficients now gives the claim.
\end{proof}

\subsection{On condition \emph{(iv)} of Theorem \ref{thm: main_zetazeros}: the distribution of spacings}
\label{subsec: spacings}

Condition \emph{(iv)} in Theorem \ref{thm: main_zetazeros} can now be stated more exactly. It is the statement that for any $K$, if $n$ is a random integer sampled uniformly from $1$ to $N$,
$$
(\zero_{n+1}-\zero_n, ..., \zero_{n+K}-\zero_n) \darrow (x_1(\varsigma^0),...,x_K(\varsigma^0)),
$$
where $\varsigma^0$ is the Palm process associated to the sine-kernel process.

\begin{remark}
There is a well-known closed formula for the distribution of $x_1(\varsigma^0)$:
$$
\PP(x_1(\varsigma^0) \leq s) = \int_0^s p_2(0,t)\, dt,
$$
where we define $p_2(0,t) = \tfrac{d^2}{dt^2} \det(1 - \mathbf{1}_{[0,t]} S \mathbf{1}_{[0,t]})$ and where $S:\,L^2(\R) \rightarrow L^2(\R)$ is given by $[Sf](x) = \int \frac{\sin \pi(x-y)}{\pi (x-y)} f(y)\, dy$, and where $\mathbf{1}_{[0,t]}:\,L^2(\R) \rightarrow L^2(\R)$ is the operator $[\mathbf{1}_{[0,t]} f](x) = \mathbf{1}_{[0,t]}(x) f(x)$.
This formula as well as formulas for the distribution of $x_{k+1}(\varsigma^0)$ for all $k$, in which the function $p_2(0,t)$ is replaced by a function $p_2(k,t)$, can be found in \cite[Eq. (6.4.31)]{Me}, though with somewhat heuristic derivations. (The parameter $2$ in $p_2(k,t)$ is used to indicate that these functions are associated to the Gaussian Unitary Ensemble. See also the last section of Odlyzko's paper~\cite{O} or Cloizeaux and Mehta's paper~\cite{CM}.) A rigorous derivation in terms of the Palm process for the distribution of $x_1(\varsigma^0)$ can be deduced from \cite[Prop 4.2.46]{AnGuZe}.
\end{remark}

\subsection{Moment bounds for zeta zeros}
\label{subsec:moment_zetazeros}

Our only input from number theory is a moment bound for counts of zeta zeros, due to A. Fujii. Set $S(t) = \tfrac{1}{\pi}\arg\, \zeta(1/2+it)$ (where as usual the argument is obtained by continuous variation along the rectilinear path joining $2$, $2+it$, $\frac12+it$, and where if $t$ is the ordinate of a zero define $S(t) = S(t+0)$), and recall the formula \cite[Thm. 9.3]{T}
\begin{equation}
\label{eq:S_errorterm}
N(t) = \frac{t}{2\pi} \log\frac{t}{2\pi} - \frac{t}{2\pi} + \frac{7}{8} + S(t)+O\Big(\frac{1}{t}\Big),
\end{equation}
for $t \geq 1$. 

\begin{theorem}[Fujii]
\label{thm:Fujii}
Let $h = h(T)$ be a bounded function of $T\geq 3$. Then for any fixed natural number $k$, we have for all $T\geq 3$,
$$
\frac{1}{T} \int_T^{2T} |S(t+h) - S(t)|^{2k}\, dt \leq C_k \log(3+ h \log T)^k,
$$
where the constant $C_k$ depends only on $k$.
\end{theorem}

We have taken $T \geq 3$ for the convenience of having $\log T \geq 1$.

\begin{proof}
This upper bound is a direct consequence of the main Theorem in \cite{F}, though the condition that $h$ must be bounded was missing there; this is corrected in \cite{F1}, from which we have taken the statement. See also \cite{MR}. Note that Fujii states this bound for $T \geq T_0$ for some large constant $T_0$, but the result remains true for $3 \leq T \leq T_0$ by compactness of this interval.
\end{proof}

\begin{corollary}
\label{cor:Fujii_cor1}
For any fixed constant $A > 0$ and any fixed $k \geq 0 $ we have for all $T\geq 3$,
$$
\frac{1}{T} \int_T^{2T} \Big| N\Big(t + \frac{A}{\log T}\Big) - N(t) \Big|^k\, dt \leq C_{A,k}
$$
where the constant $C_{A,k}$ depends only on $A$ and $k$.
\end{corollary}

\begin{proof}
	For even integers $k$ this is immediate from Theorem \ref{thm:Fujii} and \eqref{eq:S_errorterm}. For general $k$ the result follows separating the integral in two parts according to $| N(t + \frac{A}{\log T}) - N(t) |\leq1$ or $>1$.
\end{proof}

\begin{corollary}
\label{cor:Fujii_cor2}
For any fixed bounded interval $I$ and any fixed $k \geq 0$ we have for $T\geq 3$,
\begin{equation}\label{Eq:Fujii_cor2}
\frac{1}{T} \int_0^T \big| \{j:\, \zero_j-t\in I\}\big|^{k}\, dt \leq C_{I,k},
\end{equation}
where the constant $C_{I,k}$ depends only on $I$ and $k$.
\end{corollary}

\begin{proof}
The inequality \eqref{Eq:Fujii_cor2} is equivalent for $I$ or a translated interval. Therefore, we may assume that $I=[0,a]$ for some constant $a$. 
By dyadic partition \eqref{Eq:Fujii_cor2} is equivalent to prove that  for some $T_0>0$ and  $T\ge T_0$ we have
$$
J:=\frac{1}{T} \int_T^{2T} \big| \{j:\, \zero_j-t\in I\}\big|^{k}\, dt \leq C_{I,k}.
$$
Since $\vartheta(x)=\frac{x}{2}\log\frac{x}{2\pi}-\frac{x}{2}+O(1)$, there are constants $x_0$ and $0<A<B$ such that $\vartheta(Ax/\log x)\le x\le \vartheta(Bx/\log x)$ for $x>x_0$. 

Changing variables $\vartheta(x)=t$ we get 
\[J
\le \frac{1}{T}\int_{AT/\log T}^{2BT/\log(2T)}\big| \{j:\, \vartheta(\gamma_j)-\vartheta(x)\in I\}\big|^{k} \vartheta'(x)\,dx.\]
Therefore, for some constant $C$ only depending on $T_0$, we have
\[J\le C \frac{\log T}{T}\int_{AT/\log T}^{2BT/\log T}\big| \{j:\, \vartheta(\gamma_j)-\vartheta(x)\in I\}\big|^{k} \,dx.\]
The condition $\vartheta(\gamma_j)-\vartheta(x)\in I$ is equivalent to $\vartheta(x)\le \vartheta(\gamma_j)\le \vartheta(x)+a$. By the mean value theorem, for $x>AT/\log T$ this implies $x\le \gamma_j\le x+C a /\log T$, with a constant $C$ only depending on $T_0$. It follows that 
\[J\le C \frac{\log T}{T}\int_{AT/\log T}^{2BT/\log T}\bigl| \{j:\, x\le \gamma_j\le Ca/\log T\}\big|^{k} \,dx.\]
This is bounded by a constant only depending on $k$ and $a=|I|$ by Proposition 31.
\end{proof}

\subsection{Proof of Theorem \ref{thm: main_zetazeros}}
\label{subsec:proof_of_zetathm}

We may now quickly prove Theorem \ref{thm: main_zetazeros}. To see that \emph{(i)} and \emph{(ii)} are equivalent we apply Theorem \ref{thm: respace_equiv}. Setting $L(t) = \tfrac{\log t}{2\pi}$, we have $N(T) \sim T\cdot L(T)$ and $L$ is slowly varying. Moreover for $\phi(t) = \tfrac{1}{\pi} \vartheta(t)$, we have $\phi(t) \sim t L(t)$, and for $\Phi(t) = \phi(t)/t$, (see Edwards \cite[eq. (1), p.~120]{E})
$$
\frac{\Phi'(t)}{\Phi(t)} = \frac{\frac{1}{2\pi t}+O(t^{-2})}{\frac{1}{2\pi}\log\frac{t}{2\pi}-\frac{1}{2\pi}+O(1/t)}=O(1/t\log t) = o(1/t).
$$
Thus Theorem \ref{thm: respace_equiv} applies and conditions \emph{(i)} and \emph{(ii)} of Theorem \ref{thm: main_zetazeros} are equivalent.

The equivalence between \emph{(ii)} and \emph{(iii)} of Theorem \ref{thm: main_zetazeros} then follows from the estimate of Proposition \ref{prop:sinekernel_tailbound} applied to Theorem \ref{thm:tail_bound_correlations}, and in the other direction follows from the Fujii bound Corollary \ref{cor:Fujii_cor2} applied to Theorem \ref{thm:tail_bound_distributions}.

Finally the equivalence between \emph{(iii)} and \emph{(iv)} follows from Theorem \ref{thm:spacing_to_palm}, recalling the sine-kernel process is simple.

\end{document}